\documentclass[12pt]{amsart}
\usepackage{amssymb,latexsym}
\usepackage{enumerate}
\usepackage{graphicx}
\usepackage[english]{babel}
\usepackage{amssymb,amsmath}
\usepackage{cite}
\usepackage{hyperref}

\usepackage{tikz}

\newcommand*{\dashedrightarrow}[1][]{\mathbin{\tikz [baseline=-0.25ex,-latex, ->,densely dashed  ] \draw [#1] (0pt,0.5ex) -- (1.3em,0.5ex);}}%

\makeatletter
\@namedef{subjclassname@2010}{%
  \textup{2010} Mathematics Subject Classification}
\makeatother

\newtheorem{theorem}{Theorem}[section]
\newtheorem{corollary}[theorem]{Corollary}
\newtheorem{lemma}[theorem]{Lemma}

\newtheorem{conjecture}[theorem]{Conjecture}

\theoremstyle{definition}
\newtheorem{definition}[theorem]{Definition}
\newtheorem{remark}[theorem]{Remark}
\newtheorem{example}[theorem]{Example}

\numberwithin{equation}{section}

\newcommand{\fr}{{\bf F}}
\newcommand{\ff}{\mathbb F}
\newcommand{\fq}{\mathbb{F}_q}
\newcommand{\fqt}{\mathbb{F}_q(t)}
\newcommand{\frt}{\ff_q[t]}
\newcommand{\qq}{\mathbb Q}
\newcommand{\pp}{\mathbb P}
\newcommand{\af}{\mathbb A}
\newcommand{\nn}{\mathbb N}
\newcommand{\zz}{\mathbb{Z}}

\newcommand{\kk}{\mathcal K}

\newcommand{\rf}[2]{{#1 \ref{#2}}}

\newcommand{\w}{\omega}
\newcommand{\mor}{\operatorname{Mor}}

\newcommand{\ee}{\mathcal{E}}
\newcommand{\fc}{\mathcal{F}}
\newcommand{\X}{\mathcal{X}}
\newcommand{\kb}{\bar{\kappa}}
\newcommand{\vb}{\bar{V}}
\newcommand{\Conceicao}{Concei{\c c}\~ao}

\newcommand{\Spec}{\operatorname{Spec}}

\begin{document}


\baselineskip=17pt



\title[Elliptic curves over $\fqt$ with a large set of integral points]{Elliptic curves over function fields with a large set of integral points}
\author{Ricardo P. \Conceicao}
\address{Oxford College of Emory University. 100 Hamill Street, Oxford, Georgia 30054.}
\email{rconcei@emory.edu}
\subjclass[2000]{Primary 54C40, 14E20; Secondary 46E25, 20C20}

\begin{abstract}
\indent We construct isotrivial and non-isotrivial elliptic cur\-ves  over $\mathbb{F}_q(t)$ with an arbitrarily large set of separable integral points.
As an application of this construction, we prove that there are isotrivial log-general type varieties over $\fqt$ with a Zariski dense set of separable integral points. This provides a counterexample to a natural translation of the Lang-Vojta conjecture to the function field setting. We also show that our main result  provides examples of elliptic curves with an explicit and arbitrarily large set of linearly independent points.
 \end{abstract}
\maketitle

\subjclass[2010]{Primary 	11G05; Secondary 11G35 }

\keywords{Elliptic curves, integral points, Lang-Vojta conjecture, function fields}

\maketitle

\section{Introduction}\label{sec:intro}

 Inspired by the work of Caporaso,  Harry and  Mazur \cite{caporaso_uniformity_1997},  Abra\-mo\-vi\-ch \cite{abramovich_uniformity_1997} asked if there could exist a uniform bound for the number of integral points on elliptic curves over the rationals\footnote{Our discussion here holds over general number fields, but for simplicity of exposition we only give the statements over $\qq$.}. As he pointed out, this cannot be a true statement: simply choose any elliptic curve with positive rank, and make a change of coordinates to  ``clean the denominators" of an arbitrary number of rational points. This kind of construction forces a change on the integral model of the elliptic curve, and it is natural to ask whether some kind of uniformity on the number of integral points holds if  the integral model is  restricted.  Abramovich  \cite[Theorem 2]{abramovich_uniformity_1997} gives a positive answer to this question for  stably minimal models of elliptic curves over $\qq$ under the assumption of the following conjecture.

\begin{conjecture}[Lang-Vojta]\label{conj:langvojtac}
Let $X$ be a variety of  log-general type defined over  $\qq$ and let $\mathcal{X}$ be any model of $X$ over ${\zz}$. Then the set of integral points $\mathcal{X}({\zz})$ is not Zariski dense in $\mathcal{X}$.
\end{conjecture}

Here we focus on the following simpler  (but already non-trivial)  uniformity result, which is also a consequence of the Lang-Vojta conjecture.

\begin{theorem}[Section 3, \cite{abramovich_uniformity_1997}]\label{the:uniformityofquadtwistsnumber}
Let $y^2=x^3+Ax+B$ be an elliptic curve with $A$ and $B$ integers. Suppose the Lang-Vojta conjecture is true over $\qq$. Then for any square-free non-zero integer $D$, the number of integral points on the quadratic twist $Dy^2=x^3+Ax+B$ is bounded independently of $D$.
\end{theorem}

Our present work follows a long tradition in arithmetic geometry of tes\-ting the plausibility of a statement over number fields by working with an analogue statement over function fields in positive characteristic. Inspired by Theorem \ref{the:uniformityofquadtwistsnumber}, we give a negative answer to the analogue question of whether there is a uniform bound for the number of separable integral points\footnote{See Section \ref{sec:prelim} for the definition of a separable integral point.}  on quadratic twists of elliptic curves over $\fqt$.

\begin{theorem}[cf. Corollary \ref{cor:maintheorem}]\label{the:maintheorem}
Let $n$ be a positive integer. Then over $\fq(t)$ the following elliptic curves have  a separable integral point for every odd divisor of $n$: 
\begin{enumerate}
 \item $(t^{q^n}-t)y^2=x^3-x$, when $q\equiv 3\mod 4$.
 \item $y^2-y=(t^{q^n}-t)x^3$, when $q\equiv 2\mod 3$.
\end{enumerate}
In particular, if we take $n$  with a  large number of odd divisors, we obtain examples of quadratic and cubic twists of  elliptic curves with a large number of separable integral points.
\end{theorem}

These are not  the only examples of elliptic curves  for which this type of unboundedness result  holds. Theorems \ref{the:cubquarttwist} and  \ref{the:non_isot} provide  other examples of isotrivial and non-isotrivial elliptic curves  with an arbitrarily large set of separable integral points. 



In light of  Theorem \ref{the:uniformityofquadtwistsnumber} and our results on the unboundedness of the number of integral points on elliptic curves  over $\fqt$,  one could ask  whether  our construction can be used to produce counterexamples to a natural analogue    of the Lang-Vojta conjecture over function fields. We show  that this is possible in the isotrivial case.

\begin{theorem}[Cf. Section \ref{sec:langvojta}]\label{the:counterexamples}
Suppose $q\equiv 3\mod 4$. The affine variety defined over $\fqt$ by
$$
z^2=(x^3-x)(y^3-y)
$$ 
is of log-general type and has  a Zariski dense set of separable integral points.
%
\end{theorem}

In a different direction, we  present an application of our results that stems from another conjecture of Lang \cite[Conjecture 0.5]{hindry_canonical_1988}. This conjecture predicts that for a quasi-minimal model $E$ of an  elliptic curve over $\qq$, we have $|E(\zz)|<c^r$, where $c$ is a constant depending on $\qq$ and $r$ is the rank of $E(\qq)$. Consequently, one can argue that integral points on quasi-minimal model of elliptic curves over $\qq$ tend to be linearly independent. 
Here we present more evidence in support of this claim by proving  that the explicit separable integral points constructed in Theorem \ref{the:maintheorem} are linearly independent. In particular, this proves the existence of quadratic twists of supersingular elliptic curves with arbitrarily high rank over $\fqt$. This well known result was first proved  by Shafarevic and Tate \cite{tate_rank_1967}. 
 One advantage of our construction over theirs is that  we explicitly produce a large set of linearly independent points. This is  an advantage that still holds in comparison to  previous constructions 
 (see \cite{elkies_mordell-weil_1994,  ulmer_elliptic_2002, bouw_ordinary_2004,  diem_ordinary_2007, berger_towers_2008}). 
 Only recently,  Ulmer \cite{ulmer_explicit_legendre_2012,ulmer_mordell_jacobians_2012} has provided examples of explicit points on non-isotrivial elliptic curves over $\fqt$ that  generates a subgroup of the Mordell-Weill group with finite index. In a future work, we will show that our construction may also be used to produce other examples of elliptic curves with  this property.
 
We finish this introduction by saying a few words about the organization of this work. In Section \ref{sec:prelim} and \ref{sec:delsarte}, we set notation and describe the construction that is used to produce elliptic curves with an explicit rational point. Section \ref{sec:manyintegralpoints} and Section \ref{sec:non_isot} contain our examples of isotrivial and non-isotrivial  elliptic curves with a large set of  separable integral points. In Section \ref{sec:langvojta}, we give a proof of Theorem \ref{the:counterexamples}; and Section \ref{sec:largerankelliptic} deals with a proof of the unboundedness of the rank  of the elliptic curves described in Theorem \ref{the:maintheorem}.
\section{On the arithmetic of isotrivial elliptic curves}\label{sec:prelim}

{\bf Terminology:} Henceforth, $q$ denotes a power of an \emph{odd} prime $p$, $\fqt$ is a rational function field and $F'$ denotes the formal derivative of an element $F$ of $\fqt$.

\bigskip

Many of the  classical finiteness results in arithmetic geometry over number fields do not translate literally to the function field setting.  The aim of this section is to address some of the subtleties inherent to the function field setting and to establish notation.

We say that a point on an affine variety $V/\fqt$ is an \emph{integral point},  if all of its coordinates lie on $\frt$. 
In contrast with Siegel's theorem over $\qq$,  there are elliptic curves over $\fqt$ with an infinite set of integral points. 

\begin{example}\label{ex:quad_twist}
Let $D=D(t)$ be a square-free polynomial over $\fq$ and $E_D:Dy^2=f(x)$ be an elliptic curve with $f(x)$  a cubic polynomial over $\ff_q$. If $(x,y)$ is an integral point on $E_D$ then 
\begin{equation}\label{eq:frobenius_orbit}
\{(x^{q^i}, D^{\frac{q^i-1}{2}}y^{q^i}): i \in \nn\} 
\end{equation}
 is an infinite set of integral points on $E_D$.
\end{example}

Such a phenomenon needs to be taken  into account  when con\-si\-de\-ring an analogue over $\fqt$ to the uniformity result given by the conclusion of Theorem \ref{the:uniformityofquadtwistsnumber}. 
To that end, we discuss some properties of isotrivial varieties over $\fqt$.

\begin{definition}
A variety $V$ defined over $\fqt$ is said to be \emph{isotrivial} if it is isomorphic to a constant variety after a finite base extension; that is, there exists a variety $V_0$ defined over $\fq$ such that $V\cong V_0\times_{\operatorname{Spec}\fqt} \operatorname{Spec}{K}$, for some  finite extension $K/\fqt$. In this case, we say that $V$ is based on $V_0$. A variety that is not isotrivial is called \emph{non-isotrivial}.
\end{definition}


\begin{definition}
 Let $V$ be an isotrivial variety based on $V_0$. The \emph{Frobenius endomorphism} of $V$ is the map defined by
 $$
 \phi^{-1}\circ\fr\circ\phi:V\longrightarrow V,
 $$
 where $\phi:V\longrightarrow V_0$  is an isomorphism given by the isotriviality of $V$, and $\fr$ is the frobenius endomorphism of the constant variety $V_0$. 
 
 The set $\{\fr^i(P): i \in \nn\}$ is called the \emph{Frobenius orbit} of the point $P\in V$. 
\end{definition}
\begin{remark}\label{rm:frob_end}
\begin{enumerate}
 \item  Notice that the Frobenius endomorphism need not to be defined over $\fqt$ and is well defined up to conjugation by elements of $\operatorname{Aut}(V)$. 
\item In  Example \ref{ex:quad_twist}, we see that the set defined by (\ref{eq:frobenius_orbit}) is the Frobenius orbit of the point $(x,y)$ on $E_D$. 
\item We note that if $P$ is an integral point on an iso\-tri\-vial elliptic curve $E$, then $\fr(P)$ does not need to be an integral point. In fact,  let  $D=D(t)$ be a polynomial over $\fq$ and let $E/\fqt$ be the elliptic curve defined by $y^2=x^3+D$. Then $E$ is an isotrivial elliptic curve based on the elliptic curve defined over $\fq$ by $y^2=x^3+1$. It is not hard to see that the Frobenius endomorphism of $E$ is  given by $\fr(x,y)=(x^qD^{{(-q+1)}/{3}},y^qD^{{(-q+1)}/{2}})$. 
\end{enumerate}
\end{remark}

The Frobenius orbit of an integral point may or may not contain an infinite number of integral points. 
In any case,  it is true that in any isotrivial elliptic curve the number of Frobenius orbits of an integral point is finite  (see for instance section 7.3 of \cite{Mason_diophantine}). The following definition provides a more concise way of stating this finiteness result.
\begin{definition}
A point $P$ on a variety $V/\fqt$ is said to be \emph{separable} if it is not contained in the Frobenius orbit of any other point of $V$.
\end{definition}
\begin{remark}
 Notice that if $V$ is non-isotrivial then every point is a separable  point.
\end{remark}
\begin{example}\label{ex:sep_point_quad_twist}
Assume the hypotheses of Example \ref{ex:quad_twist}. Suppose $P=(x,y)$ is a point on the quadratic twist $Dy^2=f(x)$. Then $P$ is separable if, and only if, $x'\neq 0$.
\end{example}
Next we provide a simple proof of the finiteness of the number of separable integral points on  quadratic twists of constant elliptic curves. Such a result shows that it is appropriate to  ask whether the number of separable integral points on a family of isotrivial quadratic twists  can be bounded uniformly.

\begin{theorem}
\label{the:siegel_quad_twist}
Let $D=D(t)$ and $f=f(x)$ be square free polynomials defined over  $\fq$. Suppose $\deg f=3$ and  $(F,G)$ is a separable integral point on $Dy^2=f(x)$ over $\fqt$. Then $G$ divides $F'$ and $\deg D/3\leq \deg F < \deg D-1$.

In particular, the set of separable integral points on $Dy^2=f(x)$ is finite.
\end{theorem}


\begin{proof}
Since $(F,G)=(F(t),G(t))$ is a separable point, Example \ref{ex:sep_point_quad_twist} shows that $F'\neq 0$. This integral point  induces an identity on $\frt$:
\begin{equation}\label{eq1}
D(t)G(t)^2 = f(F(t)).
\end{equation}

By equating degrees in this identity we obtain the lower bound $\deg D \leq 3\deg F$. By differentiating equation (\ref{eq1}) we are led to
\begin{equation}\label{eq2}
D'(t)G(t)^2 + 2D(t)G(t)G'(t) = F'(t)f'(F(t)).
\end{equation}

Let $\beta$ be a root of $G(t)$ of multiplicity $r$. By \eqref{eq1}, we have that $(t-\beta)^r$ divides $f(F(t))$, and by (\ref{eq2}), we conclude that $(t-\beta)^r$ divides $F'(t)f'(F(t))$. Notice that $(t-\beta,f'(F(t)))=1$, since $f(x)$ has no repeated roots. Hence $(t-\beta)^r$ divides $F'(t)$ and, as a consequence, $G(t)$ divides $F'(t)$. Therefore $\deg G \leq \deg F -1$. After equating degrees in (\ref{eq1}) and using the previous inequality, we obtain the desired upper bound on $\deg F$.
\end{proof}
%


\section{Obtaining multisections on certain Elliptic surfaces}\label{sec:delsarte}


Our main objective is to construct elliptic curves containing an arbitrary number of separable integral points. 
We show in this section that an adaptation of a procedure due to  T. Shioda can be used as a first step towards this purpose.


Let $k$ be any field. In  \cite[Section 2]{shioda_explicit_1986}, Shioda considers surfaces in $\pp^3_k$ that are defined  by four monomials
\begin{eqnarray}
\X_A=\X_A(c_0,c_1,c_2,c_3):\sum_{i=0}^3 c_iX_0^{a_{i0}}X_1^{a_{i1}}X_2^{a_{i2}}X_3^{a_{i3}}=0,
\end{eqnarray}
where $A=(a_{ij})_{0\leq i,j\leq 3}$ is a $4 \times 4$ matrix with non-negative integral coefficients. Shioda calls  $\X_A$ a \emph{Delsarte surface} with matrix $A$, when it satisfies 

\begin{itemize}
\item $\det A \neq 0 \text{ in } k$;
\item $\displaystyle\sum_{j=0}^3 a_{ij}  \text{ is independent of } i = 0,1,2,3$;

\item for any $j$, some $a_{ij}=0$.
\end{itemize}

Notice that the surface defined by 
$$
c_0X_0^d+c_1X_1^d+c_2X_2^d+c_3X_3^d=0
$$
is a Delsarte surface with matrix $dI_4$, where $I_4$ is the identity matrix of dimension 4. In this case we will denote $\X_{dI_4}$ simply by $\fc_d$ and we will refer to it as a \emph{Fermat surface of degree $d$}.

For any  Delsarte surface $\X_A=\X_A(c_0,c_1,c_2,c_3)$, $A^{-1}$ is a $4 \times 4$ matrix with rational entries. If $\X_C=\X_C(c_0,c_1,c_2,c_3)$ is another Delsarte surface then there exists an integer $d$ such that $B=dA^{-1}C=(b_{ij})_{0
\leq i,j \leq 3}$ has integer entries. This easily implies the existence of a dominant rational map $ \X_{dC}  \dashedrightarrow \X_A$ defined by
$$
(X_0,X_1,X_2,X_3) \longmapsto \left(\prod_{j=0}^3X_j^{b_{0j}}, \prod_{j=0}^3X_j^{b_{1j}}, \prod_{j=0}^3X_j^{b_{2j}}, \prod_{j=0}^3X_j^{b_{3j}}\right).
$$
In particular, for any Delsarte surface $\X_A$ there exists an integer $d$, and a dominant rational map $\fc_d \dashedrightarrow \X_A$ from a Fermat surface of degree $d$.

In what follows we deal with surfaces  which are not necessarily Delsarte surfaces. Nonetheless, we can associate to them a matrix of exponents and use the above  procedure to construct explicit dominant rational maps among them. Ultimately we use these rational maps  to construct non-constant sections on certain elliptic surfaces with a fibration over $\pp^1_k$.  But before we present these examples, we  recall  briefly some facts about elliptic surfaces that will be used in our discussion.

Let $K=k(t)$ be a rational function field and $E/K$ be an elliptic curve. Associated to $E$ is a fibered elliptic surface $\pi:\ee\longrightarrow \pp_k^1$, as described in  \cite[Lecture 3]{ulmer_park_city_lect_2011}. A \emph{section} of $\ee$ is a $k$-rational morphism $P:\pp_k^1\longrightarrow \ee$ such that $ \pi\circ P: \pp_k^1 \longrightarrow \pp_k^1$ is the identity map. The set of sections of $\ee$ is in bijection  with the set of $k(t)$-rational points on $E$. 

A \emph{multisection} ${M}$ of an elliptic fibration  $\pi: \ee\longrightarrow \pp_k^1$ is an irreducible subvariety $M\subset\ee$ such that the projection map $\pi: M \longrightarrow \pp_k^1$ has non-zero degree. After a suitable finite base extension $C\longrightarrow \pp_k^1$, we have that $M_C:=M\times_{\pp_k^1} C$ is a section of $\ee\times_{\pp_k^1} C\longrightarrow C$. Notice that $M_C$ corresponds to a rational point on the base extension $E\times_{\Spec(K)} \Spec(k(C))$.

Next we apply the above discussion to several examples. For all of these examples we assume that:
\begin{itemize}
 \item  all varieties are defined over a field $k$ with $\operatorname{char}(k)\neq2, 3$; 
  \item $x,y$ and $z$ are  coordinates of $\pp^2_k$ and $t$ is the coordinate of $\af^1_k$;

 \item $d$, $r$ and $s$ are  integers satisfying $d,s>1$ and  $r=d/s$;
 \item $\fc_d$ is the Fermat surface in $\pp^3_k$ defined by $ X_0^{d}-X_1^{d} =X_2^{d}-X_3^{d}$; and
\item  for $1\leq i,j\leq 4$, $\ell_{ij}$ is the line  on $\fc_d$ of the form $\ell_{12}=(t,t,1,1)$, where  the indexes indicate the position of $t$ in the 4-tuple.
\end{itemize}


\begin{example}\label{lem2}

Let  $U_{1}$ be the closed subset  of  $\pp^2\times \af^1$ defined by the equation $(t^{d}-t)y^2z=x^3-xz^2$.  We associate to $U_1$ a matrix of exponents
$$
\left( \begin{array}{cccc}
0 & 2 & 1 & d\\
0 & 2& 1& 1\\
3 & 0&  0& 0\\
1& 0 & 2 & 0
\end{array}\right).
$$
Let $e=d-1$. By  following  Shioda's procedure we obtain a rational map $\pi:\fc_{2(d-1)} \dashedrightarrow U_1$ given by
$$
[X_0:X_1:X_2:X_3] \longmapsto \left(\left[X_0X_2^{e}X_3^{e/{2}}: X_1^dX_2^{e/{2}}: X_0X_3^{3e/2}\right], {X_0^2}/{X_1^2}\right).
$$

Let $\ee_1$ be the elliptic surface associated to (the generic fiber of) $U_1$. The image $\pi(\ell_{13})$ is the curve $([t^{d-1}:t^{{(d-3)}/{2}}:1],t^2)$ on $U_1$. Its closure on $\ee_1$  is a multisection of the fibration $\ee_1\longrightarrow\pp_k^1$. 

Suppose $d\equiv 3\mod 4$ and let $C_1$ be the projective curve  given by $u=t^2$. Thus the multisection above gives rise to a section on $\ee_1\times_{\pp_k^1} C_1$: namely, the closure on $\ee_1\times_{\pp_k^1} C_1$ of the curve $([u^{{(d-1)}/{2}}:u^{{(d-3)}/{4}}:1],u)$. Notice that this section corresponds to the integral point $(u^{{(d-1)}/{2}},u^{{(d-3)}/{4}})$ on the curve $(u^{d}-u)y^2=x^3-x$ defined over $k(u)$.
\end{example}

\begin{example}\label{3td}
Let $U_2\subset \pp^2\times \af^1$ be the closed set defined by the equation $y^2z-yz^2=(t^d-t)x^3$.  Let $e=d-1$. In this case, we obtain a rational map $\fc_{2e}\dashedrightarrow U_2$ given by
$$
[X_0:X_1:X_2:X_3] \longmapsto \left(\left[X_0^{e}X_1^{e}X_3^d: 
X_0^{3e}X_2: X_1^{3e}X_2\right], {X_2^3}/{X_3^3}\right).
$$

Let $\ee_2$ be the elliptic surface associated to $U_2$. Under the above rational map, the image of the lines $\ell_{12}$ and $\ell_{34}$ do not yield any non-constant curve on $U_2$. On the other hand,  the lines $\ell_{13}$ and $\ell_{24}$ yield the same multisection of $\ee_2$: the closure on $\ee_2$ of the curve $([t^{d-2}:t^{3(d-1)}:1],t^3)$. 

 Let $C_2$ be the projective curve  given by $u=t^3$. If $d\equiv 2\mod 3$, the closure on  $\ee_2\times_{\pp_k^1} C_2$ of the curve $([u^{{(d-2)}/{3}}:u^{d-1}:1],u)$ is a section. 
\end{example}

\begin{example}\label{3tdm1}
Consider the hypersurface $\mathcal{X}_D:X_0^s+X_1^s-X_2^d-X_3^d=0$ in the weighted projective space $\pp(r,r,1,1)$. 
Let $U_{3}\subset \pp^2\times \af^1$ be the closed set defined by $y^2z=x^3+(t^d+1)z^3$ and let $B$ be the matrix
$$
 \left(\begin{array}{rrrr}
0 & 2 & 1 & 0 \\
3 & 0 & 0 & 0 \\
0 & 0 & 3 & d \\
0 & 0 & 3 & 0
\end{array}\right)^{-1}
\left(\begin{array}{rrrr}
s & 0 & 0 & 0 \\
0 & s & 0 & 0 \\
0 & 0 & d & 0 \\
0 & 0 & 0 & d
\end{array}\right).
$$
If $6\mid s$ and (consequently) $6\mid d$ then associated to $B$ is the rational map $\mathcal{X}_{D}\dashedrightarrow U_{3}$ given by
\begin{equation}\label{cubic_rationalmapshioda}
[X_0:X_1:X_2:X_3] \longmapsto \left(\left[-X_1^{s/3}X_3^{d/6}:X_0^{s/2}:X_3^{d/2}\right],{X_2}/{X_3}\right).
\end{equation}

The image of the curve $(u^r,1,u,1)\subset \mathcal{X}_{D}$ under this rational map is the curve $([-1:u^{{d}/{2}}:1],u)$ on $U_3$. 

Another  rational curve on $\mathcal{X}_{D}$ is $(1,u^r,u,1)$. Associated to it is the curve $([-u^{{d}/{3}}:1:1],u)$ on $U_3$. 
The other two trivial  curves on $\mathcal{X}_{D}$, $(u^r,1,1,u)$ and $(1,u^r,u,1)$, do not yield  curves on $U_3$ different from the ones above.
\end{example}
\begin{example}\label{4tdm1}
Let $D$ and $\X_D$ be as in the previous example. Let $U_{4}\subset \pp^2\times \af^1$ be the closed set defined by  $y^2z=x^3-(t^d+1)xz^2$. If we assume that $4\mid s$, then the map
\begin{equation}\label{quad_rationalmapshioda}
[X_0:X_1:X_2:X_3] \longmapsto \left(\left[-X_1^{s/2}X_3^{d/{4}}:X_0^{s/2}X_1^{s/{4}}:X_3^{3d/{4}} \right],{X_2}/{X_3}\right)
\end{equation}
is a rational map   $\X_D\dashedrightarrow U_{4}$. 

Following the work done in the previous examples, we obtain the curves  $([-1:u^{{d}/{2}}:1],u)$ and $([-u^{{d}/{2}}:u^{{d}/{4}}:1],u)$ on $U_4$.
\end{example}



Our construction can also be used to explain the  explicit points  on non-isotrivial elliptic curves recently found by Ulmer \cite{ulmer_explicit_legendre_2012,ulmer_mordell_jacobians_2012}.
\begin{example}\label{ex:legendre}
For an integer $f\geq 0$, one can easily verify (see \cite[Section 3]{ulmer_explicit_legendre_2012}) that the Legendre elliptic curve $y^2=x(x+1)(x+t^{p^f+1})$ over $\fqt$ contains the integral point $P=(t,t(t+1)^{(p^f+1)/2})$ . 

To show how our procedure can be used to construct this point, first consider the closed set $U_5$ of $ \pp^2 \times \af^2$ defined by
$$
zy^2=x^3+zv^dx^2+zx^2+u^dxz^2.
$$
Notice that by setting $v=u$ in the defining equation of $U_5$, we  recover the surface $E:zy^2=x(x+z)(x+zu^d)$ in $\pp^2\times\af^1$ associated to the Legendre elliptic curve.

$U_5$ determines a matrix of exponents
$$
\left( \begin{array}{ccccc}
0  &     2 &   1    &     0 & 0\\
3     &    0  & 0     &        0& 0\\
2      &   0  & 1 & d     &        0\\
   2         &   0  &  1 &        0          & 0\\
1 & 0 & 2 & 0 & d
\end{array}\right)
$$
which, following Shioda, can be used to obtain a dominant rational map from the Fermat hypersurface $X_0^d-X_1^d-X_2^d-X_3^d-X_4^d=0$ to $U_5$. This map is given by:
$$
[X_0:X_1:X_2:X_3:X_4]\longmapsto\left(\left[X_1^dX_3^{d/2}:X_0^{d/2}X_1^d:X_3^{3d/2}\right],\left(\frac{X_2}{X_3},\frac{X_1X_4}{X_3^2}\right)\right).
$$

If we  set $u=X_2/X_3=(X_1X_4)/X_3^2=v$, we conclude that 
$$
[X_0:X_1:X_2:X_3:X_4]\longmapsto \left(\left[X_1^dX_3^{d/2}:X_0^{d/2}X_1^d:X_3^{3d/2}\right],{X_2}/{X_3}\right)
$$
defines a dominant rational map $\phi:\mathcal{S}_2\dashedrightarrow E$, where $\mathcal{S}_2$ is the surface in $\pp^4$ defined by
$$
\left\{\begin{array}{l}
X_0^d-X_1^d-X_2^d-X_3^d-X_4^d=0\\
X_1X_4=X_2X_3
\end{array}\right..
$$

When $d=p^f+1$, the surface $\mathcal{S}_2$ contains the curve $\ell:(u^d+1,u,u^d,1,u^{p^f})$. The image $\phi(\ell)$ is the curve $([u^d,u^d(u^d+1)^{d/2},1],u^d)$ on $U_5$. If $\ee_5\longrightarrow \pp^1$ is the elliptic surface associated to $E/\fqt$, then $\phi(\ell)$ corresponds to a section on $\ee_5\times_{\pp^1} C$, where $C$ is the projective curve defined by  $t=u^d$. Finally, this section yields the rational point $P$ given above.

\end{example}


\begin{remark}
Let $M$ be a multisection  of an elliptic surface $\ee\longrightarrow C$, and let $C_0\longrightarrow C$ be a base extension such that $M_0:=M\times_{C} C_0$  is a section of $ \ee\times_{C} C_0$. We  obtain a section of  $\ee\longrightarrow C$ by consi\-dering  the trace $\sum_{\sigma} (M_0)^\sigma$, where the summation is over the elements of $ \operatorname{Gal}(k(C_0)/k(C))$. In all of the previous examples, the trace of the  multisections does not yield any new information: a multisection will either be traced down to the zero section or to an integer multiple of the original section.
\end{remark}


\section{Isotrivial elliptic curves with a large set of  integral points}\label{sec:manyintegralpoints}

 In this and the following section we consider the generic fiber of the elliptic surfaces studied in Examples  \ref{lem2} through \ref{ex:legendre} and show that  over $\ff_q(t)$ these elliptic curves may have an arbitrarily large set of separable integral points. In the isotrivial case we achieve this result in  two distinct ways, which we now proceed to explain.

\subsection{A construction related to additive polynomials}
 Example \ref{lem2}  shows that over $k(t)$ the elliptic curve defined by $(t^{d}-t)y^2=x^3-x$  contains the integral point
\begin{equation}\label{eq:pointquad}
 (t^{(d-1){/2}},t^{({d-3}){/4}}), \text{ when }d\equiv 3\mod 4,
\end{equation}
 and Example \ref{3td} proves that $y^2-y=(t^{d}-t)x^3$ contains the point
\begin{equation}\label{eq:pointcubic}
 (t^{{(d-2)}/{3}},t^{d-1}), \text{ when $d\equiv 2\mod 3$}.
\end{equation}

The existence of  both of these points can be interpreted in the following way. Consider the elliptic curve $E$ defined over the function field $k(u)$ by $uy^2=x^3-x$. When $d\equiv 3\mod 4$, the point defined by \eqref{eq:pointquad} can be seen as a point on the base extension $E\times_{\Spec(k(u))} \Spec k(t)$, where $k(t)/k(u)$ is the field extension defined by $u=t^{d}-t$. By assuming this point of view, we will be able to construct extra integral points on the elliptic curves $ (t^{q^n}-t)y^2=x^3-x$  and $y^2-y=(t^{q^n}-t)x^3$ defined over $\fqt$. The points we construct are defined over certain extensions  $\fq(t)/\fq(u)$ where $u$ is transcendental over $\fq$ and $t$ satisfies $u=B(t)$, with $B(t)$ an $\fq$-additive polynomial. For this reason we present some  facts about additive polynomials that will be needed shortly. We start with their definition.

\begin{definition}
An $\ff_q$-\emph{additive polynomial} $A(t)$ is a polynomial in $\ff_q[t]$ of the form 
$$
A(t)=\sum_{i=0}^n a_it^{q^i}.
$$
We will denote the set of all $\ff_q$-additive polynomials by $\ff_q[{\bf F}]$.
\end{definition}

An additive polynomial can be seen as an $\ff_q$-polynomial  in the indeterminate ${\bf F}$, the $q$-Frobenius map $t\longmapsto t^q$. Indeed, start by defining $\fr^0(t):=t$. We have that the $i$-th self composition of $\fr$ is the polynomial $\fr^i(t)=t^{q^i}$, and so an additive polynomial $A(t)$ is the same as a $\fq$-linear combination of powers of Frobenius, $A_0(\fr)=\sum_{i=0}^n a_i\fr^i$. It turns out that the set $\ff_q[{\bf F}]$ has the structure of a ring isomorphic to $\fq[X]$. 


\begin{lemma}\label{cor:additivepol}
Let $A(t),B(t)\in\ff_q[\fr]$ and $\alpha\in\ff_q$. Then:
\begin{enumerate}
\item[(1)] $A(t)+B(t)\in\ff_q[\fr]$;
\item[(2)] $\alpha A(t)\in\ff_q[\fr]$;
\item[(3)] $ A(B(t))\in\ff_q[\fr]$.
\end{enumerate}

Furthermore, $\ff_q[\fr]$ can be endowed with a ring structure with multiplication defined by $A(t)\circ B(t):=A(B(t))$ and the association 
$$
P(X)=\sum_{i=0}^n a_iX^i \longmapsto P(\fr):=\sum_{i=0}^n a_i\fr^i
$$
is a ring isomorphism between $\ff_q[X]$ and $\ff_q[\fr]$.
\end{lemma}
\begin{proof}
(1) and (2) are trivial. To prove (3), write 
$$
A(t)=\sum_{i=0}^n a_i\fr^{i}=A_0(\fr) \text{ and } B(t)=\sum_{j=0}^m b_j\fr^{j}=B_0(\fr)
$$

From  the identity 
\begin{equation}\label{eq:identity_car_p}
 (x+y)^p=x^p+y^p,
\end{equation}
 true in any commutative ring of characteristic $p$, it follows that:
\begin{eqnarray*}
A(B(t)) & = & \displaystyle\sum_{i=0}^n a_i\fr^{i}\left(\sum_{j=0}^m b_j\fr^{j}\right)=\sum_{i=0}^n a_i\left(\sum_{j=0}^m\fr^{i} (b_j\fr^{j})\right)\\
& = &\displaystyle \sum_{i=0}^n a_i\left(\sum_{j=0}^m b_j\fr^{i+j} \right)=
\displaystyle \sum_{i=0}^n\sum_{j=0}^m  a_ib_j\fr^{i+j} \in\ff_q[\fr].
\end{eqnarray*}
Thus (3) is proved. Associativity is inherited  from $\ff_q[t]$ and the distribution law follows from \eqref{eq:identity_car_p}. So $\ff_q[\fr]$ is a ring. Notice that the latter equation also proves that $A_0(\fr)\circ B_0(\fr) = (A_0B_0)(\fr)$, and so the map $P(X)\longmapsto P(\fr)$ is multiplicative. Hence it is a ring isomorphism, since it is clearly an additive bijective map. 
\end{proof}


We are now ready to prove one of our main results  which is a generalization of Theorem \ref{the:maintheorem}:

\begin{theorem}\label{the:pmod4}
Let $m_1, m_2,\ldots,m_l$ be distinct odd positive integers. Suppose $A(t)=A_0(\fr)$ is an $\fq$-additive squarefree polynomial such that $X^{m_i}-1$ divides $A_0(X)$, for all $1\leq i\leq l$. 

Then the following elliptic curves contain at least $l$ separable integral points over $\fqt$.
\begin{enumerate}
 \item  $A(t)y^2=x^3-x$, when  $q\equiv 3\mod 4$.
 \item   $y^2-y=A(t)x^3$, when $q\equiv 2\mod 3$.
\end{enumerate}

\end{theorem}
\begin{proof}
To prove part (1), notice that under the hypothesis of the theorem we  have  $q^{m_i}\equiv 3\mod 4$. Therefore, if in \eqref{eq:pointquad} we take $d=q^{m_i}$ and $k=\fq$,  we obtain the point  
$$
P_i(u)=(u^{{(q^{m_i}-1)}/{2}},u^{{(q^{m_i}-3)}/{4}})
$$
on the twist $(u^{q^{m_i}}-u)y^2=x^3-x$ defined over $\fq(u)$.

By assumption, for each $1\leq i\leq l$ there exists a polynomial $B_i(X)\in\fq[X]$ such that $A_0(X)=(X^{m_i}-1)B_i(X)$. Under the isomorphism defined in Lemma \ref{cor:additivepol}, this equation gives us an identity:
\begin{equation}\label{eq:identity}
A(t)=B_i(\fr)^{q^{m_i}}-B_i(\fr).
\end{equation}

Let $K=\fq(t)$ be the extension of $\fq(u)$ defined by $u=B_i(\fr)$. The lift of the point $P_i$ to $K$,
\begin{equation}\label{eq:points_on_twist}
Q_i =P_i(B_i(\fr))=(B_i(\fr)^{{(q^{m_i}-1)}/{2}},B_i(\fr)^{{(q^{m_i}-3)}/{4}}),
\end{equation}
is a point on the twist  $A(t)y^2=x^3-x.$

If we let $\deg A(t)=q^a$, then the degree of the first coordinate of $Q_i$ is ${(q^a-q^{a-m_i})}/{2}$. So $Q_i$ and $Q_j$ are distinct points for  $i\neq j$. Since $A(t)$ is squarefree, we have $(B_i(\fr)^{{(q^{m_i}-1)}/{2}})'\neq 0$. Thus  it follows from example \ref{ex:sep_point_quad_twist} that  $Q_i$ is a separable integral point for all $1\leq i\leq l$.

The second part is proved in a similar way. The hypothesis of the theorem and \eqref{eq:pointcubic} give us the polynomial point $(u^{{(q^{m_i}-2)}/{3}},u^{q^{m_i}-1})$ on the elliptic curve $y^2-y=(u^{q^{m_i}}-u)x^3$. The identity \eqref{eq:identity}  implies that the integral point $S_i=(B_i(\fr)^{{(q^{m_i}-2)}/{3}},B_i(\fr)^{q^{m_i}-1})$ is on the curve $y^2-y=A(t)x^3$ defined over $K$. Observe that the degree of the second coordinate of $S_i$ is $q^a-q^{a-m_i}$, which shows  that the  $S_i$'s are all distinct.

The Frobenius orbit of a rational point $(x,y)$ on $y^2-y=A(t)x^3$ is $\{(A(t)^{{(q^{i}-1)}/{3}}x^{q^{i}},y^{q^{i}}): i \in \mathbb{N}  \}$. Thus  a point $(x_0,y_0)$ on $y^2-y=A(t)x^3$ is a separable point if $y_0'\neq 0$ . Since $(B_i(\fr)^{q^{m_i}-1})'\neq 0$, we conclude that $S_i$ is a separable integral point for all $i$.
\end{proof}

Our  main example of elliptic curves with an unbounded number of {sepa\-rable} integral points is  now an easy consequence of this result.
 
\begin{corollary}\label{cor:maintheorem}
Let $n$ be a positive integer. Then over $\fq(t)$ the following elliptic curves have  a separable integral point for every odd divisor of $n$: 
\begin{enumerate}
 \item $(t^{q^n}-t)y^2=x^3-x$, when $q\equiv 3\mod 4$.
 \item $y^2-y=(t^{q^n}-t)x^3$, when $q\equiv 2\mod 3$.
\end{enumerate}
\end{corollary}
\begin{proof}
Under the isomorphism defined in Lemma \ref{cor:additivepol}, the $\fq$-additive polynomial $A(t)=t^{q^n}-t$ corresponds to the polynomial $A_0(X)=X^n-1$. It is well known that $X^m-1$ divides $A_0(X)$ if and only if $m$ divides $n$. Therefore, in the above theorem, we can take $m_i$ to be the odd divisors of $n$.
\end{proof}

\subsection{A construction related to a surface containing many rational curves}


We now provide a  different construction of twists of elliptic curves with a large set of separable integral points. These elliptic curves are given by the generic fiber of the elliptic surfaces discussed in Examples \ref{3tdm1} and \ref{4tdm1}. Recall that in these examples, for  integers $d,s>1$  and $r$ satisfying  $r=d/s$, we defined the surface 
\begin{equation}\label{eq:surf_s}
 \mathcal{S}:X_0^{s}+X_1^{s}=X_2^{d}+X_3^{d},
\end{equation}
  in the weighted projective space $\pp_k(r,r,1,1)$. This surface contains the curves  $(t^r,1,t,1)$ and $(1,t^r,t,1)$ which, following the procedure described in Section \ref{sec:delsarte}, yields integral points on the curves $y^2=x^3+t^d+1$ and $y^2=x^3-(t^d+1)x$, when $6\mid s$ and $4\mid s$ respectively. The following sequence of results is used to prove that  $\mathcal{S}$ will contain many other  rational curves when we   specialize to the case where $s=q^m+1$, $d=q^n+1$, and $r=d/s$  are integers, and $k=\fq$. 
This is a consequence of the fact that the \emph{orthogonal group} $O_2(\fq)$ acts on the surface $\mathcal{S}$.
\begin{lemma}\label{ident2}
 Suppose $A= (a_{ij})\in O_2(\ff_q)$. Then for any non-negative integer $m$, the identity
\[
(a_{11}X+a_{12}Y)^{q^m+1}+(a_{21}X+a_{22}Y)^{q^m+1}=X^{q^m+1}+Y^{q^m+1}
\]
holds over $\ff_q[X,Y]$.
\end{lemma}

\begin{proof}
This is proved by rewriting the binomial $X^{q^m+1}+Y^{q^m+1}$ as a product of matrices.
\begin{eqnarray*} 
X^{q^m+1}+Y^{q^m+1} & =  &\left(\begin{array}{ll} X^{q^m} & Y^{q^m}  \end{array}\right)  \left(\begin{array}{ll} X \\ Y  \end{array}\right)
= \left(\begin{array}{ll} X^{q^m} & Y^{q^m}  \end{array}\right)  A^tA\left(\begin{array}{ll} X \\ Y  \end{array}\right)\\
& = & \left[A \left( \begin{array}{ll}  X^{q^m} \\ Y^{q^m}  \end{array}\right)\right]^t \left[A \left( \begin{array}{ll}  X \\ Y  \end{array}\right)\right] \\
& =  & \left(\begin{array}{ll} (a_{11}X+a_{12}Y)^{q^m}\\ (a_{21}X+a_{22}Y)^{q^m} \end{array}\right)^t \left(\begin{array}{ll} a_{11}X+a_{12}Y\\ a_{21}X+a_{22}Y \end{array}\right)\\
& =& (a_{11}X+a_{12}Y)^{q^m+1}+(a_{21}X+a_{22}Y)^{q^m+1}
\end{eqnarray*}   
\qedhere
\end{proof}




\begin{corollary}\label{shioda}
Let $n$ and $m$ be positive integers. Let $d=q^n+1$ and  $s=q^m+1$, and suppose that $r=d/s$  is an integer. Let  $\mathcal{S}/\fq$ be the surface defined by \eqref{eq:surf_s}. 
Then for any two matrices $A=(a_{ij})$ and $B=(b_{ij})$ in $O_2(\fq)$, the rational curve $(a_{11}u^r+a_{12},a_{21}u^r+a_{22},b_{11}u+b_{12},b_{21}u+b_{22})$ is contained in $\mathcal{S}$.
\end{corollary}
\begin{proof}
In the previous lemma take $X=u^r$ and $Y=1$ to obtain 
\begin{equation*}
(a_{11}u^r+a_{12})^{q^m+1} + (a_{21}u^r+a_{22})^{q^m+1}= (u^r)^{q^m+1}+1= u^{q^n+1}+1.
\end{equation*}
Another application of the previous lemma, with $X=u$, $Y=1$ and $m=n$, implies
\begin{equation*}
(b_{11}u+b_{12})^{q^n+1} + (b_{21}u+b_{22})^{q^n+1}= u^{q^n+1}+1.
\end{equation*}
So
$$
(a_{11}u^r+a_{12})^{q^m+1} + (a_{21}u^r+a_{22})^{q^m+1}=(b_{11}u+b_{12})^{q^n+1} + (b_{21}u+b_{22})^{q^n+1},
$$
as claimed.\qedhere
\end{proof}
We are ready to prove the main result of this section.

\begin{theorem}\label{the:cubquarttwist}
Let $n$ be an odd positive integer. Then over $\ff_{q^2}(t)$, the following elliptic curves have a separable  integral point for each positive divisior of $n$:
\begin{enumerate}
\item  $y^2=x^3-(t^{q^n+1}+1)x$, when $q\equiv 3\mod 4$.
\item $y^2=x^3+t^{q^n+1}+1$, when $q\equiv 2\mod 3$.
 
\end{enumerate}
\end{theorem}
\begin{proof}
Let $m$ be a divisor of $n$ and let $d=q^n+1$ and $s=q^m+1$. Then the fact that $n$ is odd implies that $r=d/s$ is an integer. 

To prove part (1), we first notice that  the assumption $q\equiv 3 \mod 4$ implies that $4\mid s$. Consequently, \eqref{quad_rationalmapshioda} defines a rational map from the surface $\mathcal{S}/\ff_{q^2}$, defined by \eqref{eq:surf_s},  to the surface in $\pp^2\times\af^1$ defined by $y^2z=x^3+(t^{q^n+1}+1)xz^2$. 


Let $B$ be the $2\times 2$ identity matrix and let  $A= \left(\begin{array}{ll} a & b \\ c & d \end{array}\right)\in O_2(\ff_{q^2})$  be such that $cd\neq 0$. Notice that such a matrix exists over $\ff_{q^2}$. From \rf{Corollary}{shioda} we obtain
the rational curve $(au^r+b,cu^r+d,u,1)\subset \mathcal{S}$. Its image under the rational map (\ref{quad_rationalmapshioda}) is the curve
$$
\left(\left[-(cu^r+d)^{({q^m+1})/{2}}:(au^r+b)^{({q^m+1})/{2}}(cu^r+d)^{({q^m+1})/{4}}: 1\right],
u\right)
$$ 
on the surface $y^2z=x^3+(t^{q^n+1}+1)xz^2$. Consequently, we obtain the point 
$$
R_m=\left(-(ct^r+d)^{({q^m+1})/{2}}, (at^r+b)^{({q^m+1})/{2}}(ct^r+d)^{({q^m+1})/{4}}
\right)
$$
on the  elliptic curve $E$ defined over $\ff_{q^2}(t)$ by $y^2=x^3-(t^{q^n+1}+1)x$. The binomial expansion of the first coordinate of $R_m$ contains the non-zero monomial $-cd^{({q^m-1})/{2}} t^r/2$, thus for each divisor $m$ of $n$ we obtain a distinct integral point on $E$. 

Let $D=t^d+1$. The Frobenius orbit of a point $(x_0,y_0)$ on $E$ is given by
$$
\left\{ \left(x_0^{q^i}/D^{(q^i-1)/2},y_0^{q^i}/D^{3(q^i-1)/4}\right): i\in\nn\right\}.
$$
Therefore $R_m$ is a separable integral point, since for any $i\in \nn$, the polynomial $D^{(q^i-1)/2}(ct^r+d)^{({q^m+1})/{2}}$ is not a $q$-th power.

The second part is proved in a similar way.  Under the assumption $q\equiv 2 \mod 3$, all the hypothesis of \rf{Corollary}{shioda} and Example \ref{3tdm1} are satisified. The image of the curve $(au^r+b,cu^r+d,u,1)\subset \mathcal{S}$ under the rational map (\ref{cubic_rationalmapshioda}) yields the integral point
$$
S_m=\left(-(ct^r+d)^{({q^m+1})/{3}},(at^r+b)^{({q^m+1})/{2}}\right)
$$
 on the curve $E_0:y^2=x^3+t^{q^n+1}+1$ over $\ff_{q^2}(t)$. The binomial expansion of $-(ct^r+d)^{{(q^m+1)}/{3}}$ shows that $S_m$  is a distinct point on $E_0$ for distinct divisors $m$ of $n$. It is easy to check, using part (3) of Remark \ref{rm:frob_end},  that $S_m$ is a separable integral point for all $m$.
\end{proof}

\section{Non-isotrivial elliptic curves with a large set of  integral points}\label{sec:non_isot}
The following theorem provides examples of non-isotrivial  elliptic curves with an arbitrarily large set of integral points.
\begin{theorem}\label{the:non_isot}
Let $n$ be a positive odd integer and  $a,b,c\in \fq$ with $ac\neq 0$. Then, for $d=q^n+1$ the elliptic curves defined over $\fqt$ by
$$
y^2=x(x+1)(x+t^d)
$$ and 
$$
y^2=(ax^2+bx+c)(cx^2+bt^dx+at^{2d})
$$
 contain an integral point for every divisor of $n$.
\end{theorem}
\begin{proof}
 Let $m$ be a divisor of $n$, then $r=(q^n+1)/(q^m+1)$ is an integer. 
 
 The first elliptic curve is the Legendre curve discussed in Example \ref{ex:legendre}. In this example  we showed that over the rational function field $\fq(u)$, the  curve $E_1:y^2=x(x+1)(x+u^{q^m+1})$ contains the point $P=(u,u(u+1)^{(p^m+1)/2})$. If we let $K=\fq(t)$ be the extension of $\fq(u)$ defined by $u=t^r$, then $E_1/K$ is defined by the equation $y^2=x(x+1)(x+t^d)$ and contains the point $(t^r,t^r(t^r+1)^{(p^m+1)/2})$.
 
 For the second elliptic curve, we let $f(x)=ax^2+bx+c$. Then for every divisor $m$ of $n$ it is easily verified that
 $$
(t^r,t^rf(t)^{(p^m+1)/2})
 $$ 
 is an integral point on $y^2=(ax^2+bx+c)(cx^2+bt^dx+at^{2d})$.
\end{proof}
\begin{remark}
Abramovich showed in  \cite[Corollary 1]{abramovich_uniformity_1997} that the existence of  a uniform  bound on the number of integral points on semistable elliptic curves over $\qq$ is a consequence of the Lang-Vojta conjecture. Our previous result shows that over function fields  no such uniform bound exists.  Indeed, for an odd $n$ with a large number of divisors, the curve  $y^2=x(x+1)(x+t^{q^n+1})$ contains a large set of integral points and is a semistable elliptic curve (see \cite[Section 7]{ulmer_explicit_legendre_2012}).
\end{remark}

\section{An isotrivial counter-example to an analogue of the Lang-Vojta conjecture over $\fqt$}\label{sec:langvojta}


Recall the statement of Theorem \ref{the:counterexamples}:  the isotrivial variety defined over $\fqt$ by
\begin{equation}\label{eq:def_loggen_counter}
\kk:z^2=(x^3-x)(y^3-y)
\end{equation}
 is of log-general type   and   has a Zariski dense set of separable integral points when $q\equiv 3\mod 4$. 
 In this section we  prove this statement.

\begin{remark}
 The arguments in this section can be adapted to prove that the variety defined over $\fqt$ by $u^3=(x^2-x)(y^2-y)(z^2-z) $ is of log-general type and, when $q\equiv 2\mod 3$, has a Zariski dense set of separable integral points. We leave the details for the reader.
\end{remark}

We start from the definition of  a log-general type variety.
\begin{definition}
Let $V$ be a variety defined over a field $k$.
\begin{itemize}
 \item Let $\bar{V}$ be  a nonsingular complete variety and $D$ be a divisor on $\bar{V}$ with simple normal crossings. We say that $\bar{V}$ is a \emph{smooth completion} of $V$ with \emph{smooth boundary} $D$, if $V=\bar{V}\backslash D$. 
 \item  Let $V$ be a variety with a smooth completion $\bar{V}$ and smooth boun\-da\-ry $D$. If  for every natural number $m$ we have $l_{\vb}(m(K_{\vb}+D))=0$, define  $\kb(V)=-\infty$. Otherwise, let  
 $$
 \kb(V)=\max_{m\in \nn}\{\dim\phi_m(\vb) \} ,
 $$
  where $\phi_m$ is the rational map associated to the divisor $m(K_{\vb}+D)$. The number $\kb(V)$ is called the \emph{logarithmic Kodaira dimension} of $V$.
 \item  $V$ is said to be of \emph{log-general type} if $\kb(V)=\dim(V)$.
\end{itemize}
\end{definition}

\begin{remark}\label{rm:kod_dimension}
The logarithmic Kodaira dimension satisfies the following properties:
\begin{enumerate}
\item $\kb(V\times W)=\kb(V)+\kb(W)$ (Theorem 11.3, \cite{iitaka_introduction_1982});
\item If $C$ is a curve and $D\subset C$ is finite set of points with $|D|\geq 3$ then  $\kb(C\backslash D)=1$ (\S 11.2(d),\cite{iitaka_introduction_1982});
\item If $f:V\longrightarrow W$ is an \'etale covering between nonsingular varieties then $\kb(V)=\kb(W)$ (Theorem 11.10, \cite{iitaka_introduction_1982}).
\end{enumerate}
 
\end{remark}

The above properties will be useful in the  proof of the first part of Theorem \ref{the:counterexamples}.

\begin{lemma}\label{cubicfamily}
$\kk$, defined as in \eqref{eq:def_loggen_counter}, is a variety of log-general type. 
\end{lemma}
\begin{proof}
To prove that $\kk$ is a log-general type variety, denote by $\bar{\kk}$ the projective completion of $\kk\subset \pp^3$ . The projection in the $x,y$ coordinates define a rational map  $\pi:\bar{\kk}\dashedrightarrow  \pp^1\times\pp^1$ such that the singular locus $S$ of $\bar{\kk}$ are contained in the  fibers over $0,1$ and $\infty$. Therefore $\bar{\kk}$ is nonsingular above $W:=(\pp^1\backslash\{0,1,\infty\})\times(\pp^1\backslash\{0,1,\infty\})$.

Let $\beta:\bar{V}\longrightarrow \bar{\kk}$ be the blow-up of $\bar{\kk}$ along $S$ and let $V:=\beta^{-1}(\bar{\kk}\backslash S)$. We can always assume that $\bar{V}$ is a smooth completion of $V$ (Theorem 7.21, \cite{iitaka_introduction_1982}). By definition, $\kk$ will be of log-general type if $V$ is. Since $\dim V=\dim \kk=2$,  we only need to show  that $\kb(V)=2$.

The restriction of $\pi$ to $\pi^{-1}(W)$ is an \'etale covering of $W$, therefore  $\pi\circ\beta:V\longrightarrow W$ is an \'etale covering. To finish the proof of the result, we use properties (1) -- (3) in Remark \ref{rm:kod_dimension} to show that
$$
\kb(V)=\kb(W)=\kb(\pp^1\backslash\{0,1,\infty\})+ \kb(\pp^1\backslash\{0,1,\infty\})=2.
$$ \qedhere
\end{proof} 
We now give a  proof of Theorem \ref{the:uniformityofquadtwistsnumber} that works for isotrivial elliptic curves. 

\begin{theorem}\label{the:vojta_uniformity}
Let $y^2=f(x)$ be an elliptic curve defined over $\fq$. If the set of separable integral points on the affine variety defined over $\fqt$ by $z^2=f(x_2)f(x_2)$  is not Zariski dense, then for any  non-zero square-free  polynomial $D\in\fq[t]$ the number of separable integral points on the quadratic twists $Dy^2=f(x)$ is bounded independently of $D$.
\end{theorem}
\begin{proof}
Let $E_1:y^2=f(x)$ and $\kk:z^2=f(x_1)f(x_2)$.

Assume that the set of separable integral points on $\kk$ are not Zariski dense. Thus there  exists a polynomial $g(x_1,x_2,z)$ with integral coefficients and prime to $z^2-f(x_1)f(x_2)$ such that all separable integral points in $\kk$ are contained in 
\[
\left\{\begin{array}{ll}
z^2=f(x_1)f(x_2)\\
g(x_1,x_2,z)=0
\end{array} \right..
\]
In the system above we  use the first equation to eliminate   from $g(x_1,x_2,z)$ powers of $z$ of order $\geq 2$. That way we find polynomials $g_0=g_0(x_1,x_2)$ and $g_1=g_1(x_1,x_2)$ such that the separable integral points on $\kk$ satisfy the equation
\begin{equation}\label{eq:sep_pnt}
g_0(x_1,x_2)+g_1(x_1,x_2)z=0.
\end{equation}
Notice that $g_0$ and $g_1$ are not  both identically zero, otherwise $g$ would be divisible by $z^2-f(x_1)f(x_2)$.  Also, we have that $\deg g_0, \deg g_1\leq \deg g$.

Let $E_D$ be the twist $Dy^2=f(x)$ of $E_1$ and $\phi:E_D\times E_D \longrightarrow \kk$ be the morphism defined by 
\begin{equation}\label{eq:map_phi}
 ((x_1,y_1),(x_2,y_2))\longmapsto (x_1,x_2,Dy_1y_2).
\end{equation}

The $q$-Frobenius action on $\kk$ is given by $(x_1,x_2,z)\longmapsto (x_1^q,x_2^q,z^q)$. Therefore, from Example \ref{ex:sep_point_quad_twist}  it follows that if $P$   is a separable integral point on $E_D\times E_D$ then the image $\phi(P)$ is a separable integral point on $\kk$.  
In particular, by \eqref{eq:sep_pnt} the separable integral points  $((x_1,y_1),(x_2,y_2))$ on  $E_D\times E_D$ satisfy the equation
\begin{equation}\label{eq:ident_sep_pnt}
g_0(x_1,x_2)+Dg_1(x_1,x_2)y_1y_2=0. 
\end{equation}

Fix $D\neq 0$. If  for every separable integral point $(x_0,y_0)\in E_D$ we have that $\overline{g}_0(X):=g_0(x_0,X)$ and $\overline{g}_1(X):=g_1(x_0,X)$ are  both identically zero as polynomials in $X$,
 then  $x_0$ is a root of a polynomial  of degree $\leq \deg g$. Indeed, in such a case $x_0$ is a root of the coefficients of $\overline{g}_0(X)$ and $\overline{g}_1(X)$. As a consequence,  the number of separable integral points on $E_D$ is bounded by $2\deg g$, independently of $D$.

Therefore we may assume that there exists a separable integral point $(x_0,y_0)\in E_D$ such that $\overline{g}_0(X)$ and $\overline{g}_1(X)$ are not both identically zero polynomials. Assume further that $y_0 \neq 0$.
Thus, from \eqref{eq:ident_sep_pnt} we find that  any other separable integral point $(x,y)\in E_D$  satisfy the polynomial equation
\begin{equation}\label{eq:rel_sep_point}
 h_0(x)+h_1(x)y=0,
\end{equation}
with $h_0(X)=\overline{g}_0(X)$ and $h_1(X)=D\overline{g}_1(X)y_0$. Note that $\deg h_0,\deg h_1 \leq \deg g $.

The number of separable integral points $(x,y)\in E_D$ that satisfy $h_1(x)=0$ is bounded above by $ 2 \deg h_1\leq 2\deg g$, and this bound  does not depend on $D$. On the other hand, if $(x,y)\in E_D$ is such that $h_1(x) \neq 0$ then by \eqref{eq:rel_sep_point} we have $y=-h_0(x)/h_1(x)$. From the equation defining $E_D$, it follows that $x$ satisfy the polynomial equation
\[
 f(x)h_1(x)^2 -Dh_0(x)^2 =0
\] 
 of degree at most $2\deg g +3$. This means that the number of separable integral points $(x,y)\in E_D$ with $h_1(x)\neq 0$ is bounded above by $4\deg g+6$. Once more we obtain an upper bound that does not depend on $D$, and the result follows.
\end{proof}

Below we provide the last ingredient  in the proof of Theorem \ref{the:counterexamples}.
\begin{corollary}
Let $\kk$ be defined as in \eqref{eq:def_loggen_counter}. If $q\equiv 3\mod 4$ then $\kk$ has a Zariski dense set of separable integral points.
\end{corollary}
\begin{proof}
Suppose that the set of separable integral points on $\kk$ is not Zariski dense.  Theorem \ref{the:vojta_uniformity}  implies that the number of separable integral points on the family of quadratic twists $Dy^2=x^3-x$ remains bounded as $D$ runs through square-free polynomials over $\fq$. But when $q\equiv 3\mod 4$, this contradicts Theorem \ref{the:maintheorem} and  completes the proof.
\end{proof}

\section{Elliptic curves with an explicit large set of linearly independent points} \label{sec:largerankelliptic}

In this section we prove that the points found in Theorem \ref{the:pmod4} are linearly independent.




\begin{theorem}\label{the:li} 
 Let $m_1, m_2,\ldots,m_l$ be distinct odd positive integers. Suppose $A(t)=A_0(\fr)$ is an $\fq$-additive squarefree polynomial such that $X^{m_i}-1$ divides $A_0(X)$, for all $1\leq i\leq l$. Let $E_A$ be the elliptic curve defined by $A(t)y^2=x^3-x$. Suppose  $q\equiv 3\mod 4$.

 Then the points $\{Q_1,\ldots,Q_l\}\subset E_A(\fq(t))$ defined by \eqref{eq:points_on_twist} are $\zz$-linearly independent.
\end{theorem}
\begin{proof}
 Let $C/\fq$  be the smooth projective curve defined by  $s^2=A(t)$, and let $L=\fq(C)$ be its function field. Let  $E/\fq$ be the elliptic curve defined by $y^2=x^3-x$ and $E_A/\fqt$ be the elliptic curve defined by $A(t)y^2=x^3-x$.  Notice that $E_A$ and $E$ are isomorphic over $L$ via the isomorphism
\begin{equation*}
(x,y)\longmapsto (x,sy).
\end{equation*}

The set $\mor_{\fq}(C,E)$  of $\fq$-morphisms from $C$ to $E$ is an abelian group canonically isomorphic to the Mordell-Weil group $E(L)$ (see \cite[Proposition 6.1]{ulmer_park_city_lect_2011}).

For $P=(F,G)\in E_A(\fqt)$, we let $ \phi_P:C\longrightarrow E$ be the $\fq$-morphism  $\phi_P(t,s)=(F(t),sG(t))$. As a consequence of the above discussion, the map
\begin{equation}
 \begin{array}{rll} \label{eq:intmap}
\Gamma:E_A(\fqt)&\longrightarrow &\mor_{\fq}(C,E)\\
 P&\longmapsto&\phi_P
 \end{array}
 \end{equation}
is an injective group homomorphism.


Let $Q_i=(F_i,G_i)$ be given  as in \eqref{eq:points_on_twist} and let $\phi_i=\Gamma(Q_i)$. We are left to prove that the set $\{\phi_i\}\subset \mor_{\fq}(C,E)$  is $\zz$-linearly independent.  First notice that $B_i(\fr)$, given as in \eqref{eq:identity}, is a square-free $\fq$-additive polynomial. Consequently,  there exists $\beta_i\in\ff_q^*$ such that $F'_i=\beta_i G^2_i$. Let $\w_E=dx/y$ be the invariant differential on $E$ and let $\w_C$ be the non-zero differential $dt/s$  on $C$.
Thus
\[
 \phi_i^*(\w_E)=\frac{d\phi^*_i(x)}{\phi^*_i(y)}=\frac{F'_idt}{sG_i}=\beta_i G_i\w_C.
\]

Since  the $G_i$'s   have distinct degrees,  we have that $\{\phi_i^*(\w_E)/\w_C\}\subset L$ is an $\fq$-linearly independent set. The fact that  $E$ is supersingular \cite[Example 4.5]{silverman_arithmetic_1986} allows us to use the lemma below to finish the proof of our result. 
\end{proof}

\begin{lemma}
 Let $C$ be a smooth projective curve and $E$ be a supersingular elliptic curve, both defined over $\ff_q$. Let $\w_E$ and $\w_C$ be non-zero differentials on $E$ and $C$, respectively. Let $\{\phi_i\}_{i=1}^n$ be a subset of $\mor_{\fq}(C,E)$, the set of $\fq$-morphisms from $C$ to $E$.
 
If $\{\phi_i^*(\w_E)/\w_C\}^n_{i=1}$ is  an $\ff_q$-linearly independent set in  $\ff_q(C)$ then $\{\phi_i\}_{i=1}^n$ is a set of $\mathbb{Z}$-linearly independent morphisms in $\mor_{\fq}(C,E)$.
\end{lemma}
\begin{proof}
 For an integer $m$, let $[m]$ denote the multiplication-by-$m$ map on $E$. Suppose, by contradiction, that there exists a non-trivial $\mathbb{Z}$-linear combination $\sum_{i=1}^n [a_i]\phi_i=O$. Let $p^j$ be the largest power of $p$ that divides $a_i$, for  $1\leq i\leq n$. Then
\[
[p^j]\left(\sum_{i=1}^n [b_i]\phi_i\right)=O,
\]
where $b_i={a_i}/{p^j}$. The $p$-torsion group of a supersingular elliptic curve is trivial \cite[Theorem 3.1]{silverman_arithmetic_1986}, therefore  $\sum_{i=1}^n [b_i]\phi_i=O$ is a $\mathbb{Z}$-linear combination of the $\phi_i$'s with at least one of its coefficients prime to $p$, say $b_0$. 

The linearity of the pullback of differentials (see \cite[Theorem 5.2]{silverman_arithmetic_1986} for a proof of this fact when $C$ is an elliptic curve) implies that
\[
 0=\left(\sum_{i=1}^n[b_i]\phi_i\right)^*(\w_E)=\sum_{i=1}^n (\phi_i^*b_i^*)(\w_E)=\sum_{i=1}^n b_i\phi_i^*(\w_E).
\]
Hence
\[
 \sum_{i=1}^n b_i\frac{\phi_i^*(\w_E)}{\w_C}=0.
\]
By assumption $\{\phi_i^*(\w_E)/\w_C\}^n_{i=1}$ is an $\fq$-linearly independent set. Therefore   $p\mid b_i$ for all $i$. But this contradicts the fact that $p$ is prime to $b_0$, and the result follows.
\end{proof}

\begin{remark}
 Notice that a similar argument can  be used to prove the linear independency of the explicit points on the cubic twists $A(t)x^3=y^2-y$ found in Theorem \ref{the:pmod4}. 
\end{remark}

\section*{Acknowledgments}

I would like to thank Jos\'e Felipe Voloch for his invaluable support and guidance during the completion of this work. I am also thankful to  Douglas Ulmer  for several discussions about the topics in this article and, in particular, for introducing me to the elliptic curves in Theorem \ref{the:non_isot}.

\bibliography{twists}{}
\bibliographystyle{abbrv}

\end{document}